\documentclass[a4paper, 10pt,reqno]{amsart}
\title[Green's function in the Heisenberg Group]{Green's function for certain domains in the Heisenberg Group $\mathbb H_n$}
\usepackage[ansinew]{inputenc}
\usepackage{yfonts}
\usepackage{calligra}
\usepackage{amscd}
\usepackage{amssymb, latexsym}
\usepackage{mathrsfs}
\usepackage{fancyhdr}
\usepackage{enumerate}
\usepackage{textcomp, phonetic}
\usepackage{setspace}
\usepackage{latexsym}
\def\H{\mathbb H}
\def\C{\mathbb C}
\def\R{\mathbb R}
\def\ben{\begin{eqnarray*}}
\def\een{\end{eqnarray*}}
\newtheorem{thm}{Theorem}[section]

\author[S. Dubey, A. Kumar And M. M. Mishra]{ Shivani Dubey, Ajay Kumar* and Mukund Madhav Mishra}
\address{Department of Mathematics, University of Delhi, Delhi, India}
\thanks{*Corresponding author. Email: akumar@maths.du.ac.in}
\date{}
\begin{document}
\begin{abstract}
We obtain explicit smooth Green's functions for annular domain and infinite strip by using kelvin $R$-transform in the Heisenberg group $\H_n$.
\end{abstract}
\maketitle
\section{Introduction}
A Green's fuction is an integral Kernel that can be used to solve inhomogeneous differential equations with boundary conditions. It has interesting physical significances when the involved differential operator is a Laplacian. For example, for the heat conduction equation, the Green's function is proportional to the temperature caused by a concentrated energy source. The rich geometric structure of Heisenberg group allows us to construct explicit examples of domains that are relevant in Potential theory. On the Heisenberg group we have an analogue of the Laplacian which was first studied by Folland and Stein \cite{FM1}. The study of Green's function on the Heisenberg group became interesting after Folland \cite{F2} found a smooth fundamental solution for this operator. Kor\'{a}nyi first gave a Green's function for circular data for a certain gauge ball \cite{kor3} using the Kelvin transform on $\mathbb{H}_n$ \cite{kor4,GK5}. Annular domain in the Heisenberg group was studied in \cite{KM7}. Green's functions for polyharmonic functions for the domains such as disc, half-space and ring in the complex plane have been studied in \cite{BV8,HT9}. In \cite[p. 386]{CH10}, Courant and Hilbert gave the Green's function for annular domain in the classical case by infinitely many reflection of pole with respect to boundary of ball. In this article, we have generalized the method of Courant and Hilbert to the case of annular domain in the Heisenberg group. The role of repeated reflections here is being played by repeated Kelvin transforms. The same idea easily worked for the case of an infinite strip in the Heisenberg group.\\

The Heisenberg group $\mathbb{H}_n$ is the set of points $[z,t]\in \mathbb{C}^n\times\mathbb{R}$ with the multiplication given by
$$[z,t].[z',t']=[z+z',t+t'+2imz.\bar{z}'],$$ $z,z'\in\C^n,t,t'\in\R.$\\
The basis of the Lie algebra of $\mathbb{H}_n$ is $\{Z_j,\bar{Z}_j,T:1\leq j\leq n\}$ where 
\begin{eqnarray*}
Z_j &=& \partial_{z_j}+i\bar{z}_j\partial_t ;\\
\bar{Z}_j &=& \partial_{\bar{z}_j}-iz_j\partial_t ;\\
T &=& \partial_t.
\end{eqnarray*}
The sublaplacian on $\mathbb{H}_n$ is  given by
$$\Delta_0=2\sum_{j=1}^n(\bar{Z_j}Z_j+Z_j\bar{Z_j}).$$
We shall consider a slightly modified subelliptic operator $L_0=-\frac{1}{4}\Delta_0$.
The natural gauge on $\mathbb{H}_n$ is given by

$$N(z,t)=({|z|}^4+t^2)^\frac{1}{4}.$$
The fundamental solution for $L_o$ on  $\mathbb{H}_n$ with pole at identity is given in \cite{F2} as 
$$g_e(\xi)=g_e([z,t])=a_o(|z|^4+t^2)^{-\frac{n}{2}},$$
where
$$a_o=2^{n-2}\frac{(\Gamma(\frac{n}{2}))^2}{\pi^{n+1}},$$
and $\xi=[z,t]$. The fundamental solution with pole at $\eta$ is given by

$$g_\eta(\xi)=g_e(\xi^{-1}\eta).$$
From \cite{kor6}, for $\eta=[\varsigma,\tau]$ and $\xi=[z,t]$,

$$g_\eta(\xi)=a_o|C(\eta,\xi)-P(\eta,\xi)|^{-n},$$
where

$C(\eta,\xi)=|z|^2+|\varsigma|^2+i(t-\tau)$  and $P(\eta,\xi)=2z.\bar{\varsigma}.$\\
For an integrable function $f$ on $\mathbb{H}_n$, we denote the average of $f$ by

$$\bar{f}([z,t])=\frac{1}{2\pi}\int_0^{2\pi}f([e^{i\theta}z,t])d\theta.$$
As in \cite{kor6},

$$g_\eta(\xi)=a_o|C(\eta,\xi)|^{-n}F\left(\frac{n}{2},\frac{n}{2};n;\frac{|P(\eta,\xi)|^2}{|C(\eta,\xi)|^2}\right),$$
where $F$ is the Gaussian hypergeometric function.\\
\section{GREEN'S FUNCTION FOR ANNULAR DOMAIN}
In this section, $D$ will denote the annulus $\{\xi\in\H_n: 0<R<N(\xi)<1\}$.\\The Kelvin transform on the Heisenberg group has been defined and studied in \cite{kor4}. For any $f$ on $\H_n$ the Kelvin transform of $f$ is defined by 
$$Kf=N^{-2n}f\circ h,$$ where h is the inversion,
$$h([z,t])=\left[\frac{-z}{|z|^2-it},\frac{-t}{|z|^4+t^2}\right],$$ for $[z,t] \in\H_n\setminus\{e\}$. This transform sends a harmonic function on $\H_n\setminus\{e\}$ to a harmonic function. It was shown in \cite{kor3} that for a circular function $f$ on  $\H_n\setminus\{e\}$, we have
$$K(f)(\xi^{-1})=f(\xi),$$ 
for all $\xi\in \H_n\setminus\{e\}$ with $N(\xi)$=1.\\ From \cite[(3.3)]{kor3} we have, for $\eta\neq e\in\H_n$
\begin{eqnarray}
K(g_\eta)=N(\eta)^{-2n}g_{\eta^*},
\end{eqnarray}
where we wrote $\eta^*$ for $h(\eta)$.\\The Kelvin R-transform on the Heisenberg group was defined and studied in \cite{KM7}. For $f$ defined on
$\H_n\setminus\{e\}$, the Kelvin R-transform is defined as
$$K_R(f)=R^{2n}g_ef\circ h_R,$$\\ where $h_R$ is the inversion with respect to the Kor\'{a}nyi ball of radius R, i.e,\\ $\{[z,t]:N(z,t)<R\}$
$$h_R([z,t])=\left[\frac{-R^2z}{|z|^2-it},\frac{-R^4t}{|z|^4+t^2}\right],$$ for [z,t] $\in\H_n\setminus\{e\}$.\\This transform sends a harmonic function on $\H_n\setminus\{e\}$ to a harmonic function. It was shown in \cite{KM7} that for a circular function $f$ on $\H_n\setminus\{e\}$, and $R>0$, we have 
$$K_R(f)(\xi^{-1})=f(\xi),$$ for all $\xi\in\H_n$ with $N(\xi)=R$.\\From \cite[(14)]{KM7} we have, for $\eta\neq e \in\H_n$
\begin{eqnarray}
K_R(g_\eta)=R^{2n} N(\eta)^{-2n}g_{\eta^+},
\end{eqnarray}
where $\eta^+=h_R(\eta).$\\A Green function for annular domain was given in \cite{KM7}, however the function was not continuous. The Green function constructed below is smooth. The construction below is dependent on infinite reflections across $D$ using Kelvin R-transform. Consequently a solution to the Dirichlet boundary value problem on $D$ has been provided.\\ We define functions $H_k(\eta,\xi),M_k(\eta,\xi),U_k(\eta,\xi)$ and $V_k(\eta,\xi)$ inductively. Define
\ben
H_1(\eta,\xi)&=&K(\bar{g_\eta})\circ i,\\
M_1(\eta,\xi)&=&\bar{g_\eta}(\xi),\\
U_1(\eta,\xi)&=&K_R(\bar{g_\eta})\circ i,\\
V_1(\eta,\xi)&=&K(K_R(\bar{g_\eta})\circ i)\circ i.
\een When $H_k(\eta,\xi)$, $M_k(\eta,\xi)$, $U_k(\eta,\xi)$, $V_k(\eta,\xi)$ are defined, define
\ben
H_{k+1}(\eta,\xi)&=&K(K_R(H_k(\eta,\xi))\circ i)\circ i,\\
M_{k+1}(\eta,\xi)&=&K_R(K(M_k(\eta,\xi))\circ i)\circ i,\\
U_{k+1}(\eta,\xi)&=&K_R(K(U_k(\eta,\xi))\circ i)\circ i,\\ 
V_{k+1}(\eta,\xi)&=&K(K_R(V_k(\eta,\xi))\circ i)\circ i,
\een where $i$ denotes inversion in the Heisenberg group i.e, $i[z,t]=[-z,-t]$ for $[z,t]\in \H_n$.\\ We claim that
\begin{eqnarray}
\sum_{k=1}^{\infty}{[M_K(\eta,\xi)-H_k(\eta,\xi)]}+\sum_{k=1}^{\infty}{[V_k(\eta,\xi)-U_k(\eta,\xi)]}
\end{eqnarray}
is absolutely and uniformly convergent.\\ We first show that\\
\begin{eqnarray}
H_k(\eta,\xi)&=&R^{(2k-2)n}(N(\xi))^{-2n}\left|\frac{|z|^2}{|z|^4+t^2}+R^{2(2k-2)}|\varsigma|^2+i\left(\frac{t}{|z|^4+t^2}-R^{2(2k-2)}\tau\right)\right|^{-n}\nonumber\\
&&\times F\left(\frac{n}{2},\frac{n}{2};n;\frac{u_{k,1}}{v_{k,1}}\right),
\end{eqnarray}
where $$\frac{u_{k,1}}{v_{k,1}}=\frac{4R^{2(2k-2)}|z|^2|\varsigma|^2}{1+R^{4(2k-2)}(|\varsigma|^4+\tau^2)(|z|^4+t^2)+2R^{2(2k-2)}(|z|^2|\varsigma|^2-t\tau)}.$$
We prove $(4)$ by induction on $k$.\\ For $k=1$, $ H_1(\eta,\xi)=K(\bar{g_\eta})(-\xi).$\\Assume $(4)$ for $k=l$, we show that the validity of $(4)$ for $k=l+1$.\\ From $(1)$ and $(2)$, we have
\begin{eqnarray*}
H_{l+1}(\eta,\xi)&=&K(K_R(H_l(\eta,\xi))\circ i)(-\xi)\\
&=&g_e(-\xi)K_R(H_l(\eta,\xi))(h\xi)\\
&=&R^{2n}g_e(-\xi)g_e(h\xi)H_l(\eta,\xi)(h_Rh\xi)\\
&=&R^{2n}H_l(\eta,\xi)(R^2z,R^4t)\\
&=&R^{2n}R^{(2l-2)n}(N(\xi))^{-2n}R^{-4n}\left|\frac{R^4|z|^2}{R^8(|z|^4+t^2)}+R^{2(2l-2)}|\varsigma|^2
+i\left(\frac{R^4t}{R^8(|z|^4+t^2)}-R^{2(2l-2)}\tau\right)\right|^{-n} \\
&&\times F\left(\frac{n}{2},\frac{n}{2};n;\frac{a_{l+1}}{b_{l+1}}\right)\\
&=&R^{2ln}(N(\xi))^{-2n}\left|\frac{|z|^2}{|z|^4+t^2}+R^{2(2l)}|\varsigma|^2+i\left(\frac{t}{|z|^4+t^2}-R^{2(2l)}\tau\right)\right|^{-n} F\left(\frac{n}{2},\frac{n}{2};n;\frac{u_{l+1,1}} {v_{l+1,1}}\right),
\end{eqnarray*} 
where
 $$\frac{u_{l+1,1}}{v_{l+1,1}}=\frac{4R^{2(2l)}|z|^2|\varsigma|^2}{1+R^{4(2l)}(|\varsigma|^4+\tau^2)(|z|^4+t^2)+2R^{2(2l)}(|z|^2|\varsigma|^2-t\tau)}.$$\\
Therefore, by Induction $(4)$ follows.\\Similarly, we have\\
$M_k(\eta,\xi)=R^{(2k-2)n}||z|^2+R^{2(2k-2)}|\varsigma|^2+i(t-R^{2(2k-2)}\tau)|^{-n} F\left(\frac{n}{2},\frac{n}{2};n;\frac{u_{k,2}}{v_{k,2}}\right),$\\where\\ $$\frac{u_{k,2}}{v_{k,2}}=\frac{4R^{2(2k-2)}|z|^2|\varsigma|^2}{(|z|^4+t^2)+R^{4(2k-2)}(|\varsigma|^4+\tau^2)+2R^{2(2k-2)}(|z|^2|\varsigma|^2-t\tau)},$$
$U_k(\eta,\xi)=R^{2kn}(N(\xi))^{-2n}\left|\frac{R^{4k}|z|^2}{|z|^4+t^2}+|\varsigma|^2+i\left(\frac{R^{4k}t}{|z|^4+t^2}-\tau\right)\right|^{-n} F\left(\frac{n}{2},\frac{n}{2};n;\frac{u_{k,3}}{v_{k,3}}\right),$\\where\\ $$\frac{u_{k,3}}{v_{k,3}}=\frac{4R^{2(2k)}|z|^2|\varsigma|^2}{R^{2(4k)}+(|\varsigma|^4+\tau^2)(|z|^4+t^2)+2R^{2(4k)}(|z|^2|\varsigma|^2-t\tau)},$$
$V_k(\eta,\xi)=R^{2kn}|R^{4k}|z|^2+|\varsigma|^2+i(R^{4k}t-\tau)|^{-n} F\left(\frac{n}{2},\frac{n}{2};n;\frac{u_{k,4}}{v_{k,4}}\right),$\\where\\ $$\frac{u_{k,4}}{v_{k,4}}=\frac{4R^{2(2k)}|z|^2|\varsigma|^2}{R^{2(4k)}(|z|^4+t^2)+(|\varsigma|^4+\tau^2)+2R^{2(4k)}(|z|^2|\varsigma|^2-t\tau)}.$$\\Now, we will prove that the Gaussian Hypergeometric functions involved in expression of infinite series $(3)$ are uniformly bounded.\\Consider
\ben
&&1+R^{4(2k-2)}(|\varsigma|^4+\tau^2)(|z|^4+t^2)+2R^{2(2k-2)}(|z|^2|\varsigma|^2-t\tau)-4R^{2(2k-2)}|z|^2|\varsigma|^2\\
&&=1+R^{2(2k-2)}[R^{2(2k-2)}(|\varsigma|^4+\tau^2)(|z|^4+t^2)-2|z|^2|\varsigma|^2-2t\tau]\\
&&\geq 1-2R^{2(2k-2)}[|z|^2|\varsigma|^2+t\tau]\\
&&\geq 1-4R^{2(2k-2)}.
\een We can choose $k$ large enough such that $R^{2(2k-2)}<\frac{1}{8}$ so that argument $\frac{u_{k,1}}{v_{k,1}}$ of $F\left(\frac{n}{2},\frac{n}{2};n;\frac{u_{k,1}}{v_{k,1}}\right)$ is bounded away from $1$. Thus $\left\{F\left(\frac{n}{2},\frac{n}{2};n;\frac{u_{k,1}}{v_{k,1}}\right)\right\}_{k=1}^\infty$ is uniformly bounded for $\eta\neq\xi$, say
$$\left|F\left(\frac{n}{2},\frac{n}{2};n;\frac{u_{k,1}}{v_{k,1}}\right)\right|<E_1.$$\\Similarly we have constants $E_2, E_3, E_4$ such that
$$\left|F\left(\frac{n}{2},\frac{n}{2};n;\frac{u_{k, i}}{v_{k,i}}\right)\right|<E_i \;\text{for}\;i=2, 3, 4.$$ We assert that both series in $(3)$ are uniformly and absolutely convergent on compact neighbourhoods of $\xi$.\\We have 
\ben
|M_k(\eta,\xi)|&=&|R^{2kn}|\left||z|^2+R^{4k}|\varsigma|^2+i(t-R^{4k}\tau)\right|^{-n}\left|F\left(\frac{n}{2},\frac{n}{2};n;\frac{u_{k,2}}{v_{k,2}}\right)\right|\\
&\leq&\left[2(1+|z|^2|\varsigma|^2+|t\tau|)\right]^{\frac{-n}{2}}.E_2.|R^{2kn}|.
\een\\Since $R<1$, the series $\sum_{k=1}^\infty{R^{2kn}}$ is convergent and so $\sum_{k=1}^\infty{|M_k(\eta,\xi)|}$ is uniformly convergent on compact neighbourhood of $\xi=[z,t]$.\\Similar estimates show that $\sum_{k=1}^\infty{|H_k(\eta,\xi)|}$ is uniformly convergent on compact neighbourhood of $\xi=[z,t]$.\\Hence $\sum_{k=1}^\infty{[M_k(\eta,\xi)-H_k(\eta,\xi)]}$ is absolutely and uniformly convergent on compact neighbourhoods of $\xi$ (for$\;\eta\neq\xi$).\\We have,\\
\ben
|V_k(\eta,\xi)|&=&|R^{2kn}|\left|R^{4k}|z|^2+|\varsigma|^2+i(R^{4k}t-\tau)\right|^{-n}\left|F\left(\frac{n}{2},\frac{n}{2};n;\frac{u_{k,4}}{v_{k,4}}\right)\right|\\
&\leq&\left[2(1+|z|^2|\varsigma|^2+|t\tau|)\right]^{\frac{-n}{2}}.E_4.|R^{2kn}|.
\een\\Since $R<1$, the series $\sum_{k=1}^\infty{R^{2kn}}$ is convergent and so $\sum_{k=1}^\infty{|V_k(\eta,\xi)|}$ is uniformly convergent on compact neighbourhood of $\xi=[z,t]$.\\Similar estimates show that  $\sum_{k=1}^\infty{|U_k(\eta,\xi)|}$ is uniformly convergent on compact neighbourhood of $\xi=[z,t]$.\\Hence $\sum_{k=1}^\infty{[V_k(\eta,\xi)-U_k(\eta,\xi)]}$ is absolutely and uniformly convergent on compact neighbourhoods of $\xi$ (for$\;\eta\neq\xi$).\\Therefore, for each $\eta$,
\begin{eqnarray}
G(\eta,\xi)=\sum_{k=1}^{\infty}{[M_k(\eta,\xi)-H_k(\eta,\xi)]}+\sum_{k=1}^{\infty}{[V_k(\eta,\xi)-U_k(\eta,\xi)]},
\end{eqnarray} is a well defined function for $\xi\neq\eta$.\\Next, we show that $G(\eta,\xi)$ works as a Green's function when applied to circular functions.\\
\begin{thm}
The function $G(\eta,\xi)$ is a smooth function on $D=\{\xi\in\H_n:0<R<N(\xi)<1\}$ and satisfies the following.\\(i) $L_0G(\eta,\xi)=\delta_\eta$.\\(ii) Limits of the function $G(\eta,\xi)$ vanishes at the boundaries of annular domain i.e, at $N(\xi)=1$ and $N(\xi)=R$.
\end{thm}

\begin{proof}
First note that $H_k(\eta,\xi)$, $M_k(\eta,\xi)$, $k>1$, $U_k(\eta,\xi)$ and $V_k(\eta,\xi)$ are all harmonic functions on $D$ (this follows from definition of these functions and properties of the Kelvin transforms $K$ and $K_R$). Since the series are absolutely and uniformly convergent on compact sets so Laplacian can be applied to series term by term
\ben
\sum_{k=1}^\infty{L_0[M_k(\eta,\xi)-H_k(\eta,\xi)]}+\sum_{k=1}^\infty{L_0[V_k(\eta,\xi)-U_k(\eta,\xi)]}&=&L_0 M_1(\eta,\xi)\\
&=&L_0\bar{g_\eta}(\xi).
\een Thus,
$$L_0G(\eta,\xi)=L_0\bar{g_\eta}(\xi)=\delta_\eta.$$ 
It can be easily seen that as $N(\xi)\rightarrow 1$, $M_k(\eta,\xi)\rightarrow H_k(\eta,\xi)$ and $V_k(\eta,\xi)\rightarrow U_k(\eta,\xi)$. Therefore,$$\lim_{N(\xi)\rightarrow 1}\left(\sum_{k=1}^\infty{[M_k(\eta,\xi)-H_k(\eta,\xi)]}+\sum_{k=1}^\infty{[V_k(\eta,\xi)-U_k(\eta,\xi)]}\right)=0.$$ And,\\
\ben
&&\lim_{N(\xi)\rightarrow R}\left(\sum_{k=1}^\infty{[M_k(\eta,\xi)-H_k(\eta,\xi)]}+\sum_{k=1}^\infty{[V_k(\eta,\xi)-U_k(\eta,\xi)]}\right)\\
&&=\lim_{N(\xi)\rightarrow R}\left((M_1(\eta,\xi)-H_1(\eta,\xi))+(M_2(\eta,\xi)-H_2(\eta,\xi))\right.\\
&&\left.+(M_3(\eta,\xi)-H_3(\eta,\xi))+\ldots+(V_1(\eta,\xi)-U_1(\eta,\xi))\right.\\
&&\left.+(V_2(\eta,\xi)-U_2(\eta,\xi))+(V_3(\eta,\xi)-U_3(\eta,\xi))+\ldots\right)\\
&&=\lim_{N(\xi)\rightarrow R}\left((M_1(\eta,\xi)-U_1(\eta,\xi))+(M_2(\eta,\xi)-H_1(\eta,\xi))\right.\\
&&\left.+(M_3(\eta,\xi)-H_2(\eta,\xi))+\ldots+(V_1(\eta,\xi)-U_2(\eta,\xi))\right.\\
&&\left.+(V_2(\eta,\xi)-U_3(\eta,\xi))+(V_3(\eta,\xi)-U_4(\eta,\xi))+\ldots\right)\\
&&=\lim_{N(\xi)\rightarrow R}(M_1(\eta,\xi)-U_1(\eta,\xi))+\lim_{N(\xi)\rightarrow R}(M_2(\eta,\xi)-H_1(\eta,\xi))\\
&&+\lim_{N(\xi)\rightarrow R}(M_3(\eta,\xi)-H_2(\eta,\xi))+\ldots+\lim_{N(\xi)\rightarrow R}(V_1(\eta,\xi)-U_2(\eta,\xi))\\
&&+\lim_{N(\xi)\rightarrow R}(V_2(\eta,\xi)-U_3(\eta,\xi))+\lim_{N(\xi)\rightarrow R}(V_3(\eta,\xi)-U_4(\eta,\xi))+\ldots.
\een
We have, $M_1(\eta,\xi)=\bar{g_\eta}(\xi)$ and $U_1(\eta,\xi)=K_R(\bar{g_\eta})(-\xi)$\\Therefore, by $(2)$, $M_1(\eta,\xi)-U_1(\eta,\xi)=0$. Similarly, by using properties of $K_R$ on each and every term, all terms of this series are equal to zero.\\Hence, the function $G(\eta,\xi)$ given in $(5)$ is a smooth function and is a Green's function for $D$ when applied to circular functions.
\end{proof} The Poisson kernel is the normal derivative of Green's function and, from \cite{kor3}, is given by
$$P(\eta,\xi)=-\frac{1}{4}\frac{\partial}{\partial n_0} G(\eta,\xi),\;\xi\in\partial D,$$ where\\
$\frac{\partial}{\partial n_0}=\begin{cases}
\frac{1}{|z|}(\bar{A}E+A\bar{E})&\text{at$(\partial D)_1$ i.e, at the boundary of $\{\xi\in\H_n:N(\xi)=1\}$}\\
\frac{-1}{R^2|z|}(\bar{A}E+A\bar{E})&\text{at $(\partial D)_2$ i.e, at the boundary of $\{\xi\in\H_n:N(\xi)=R\}$},
\end{cases}$\\ 
$A=|z|^2-it$ and $E=\sum{z_j Z_j}.$ \\ An easy calculation using properties of the Hypergeometric function shows that $P(\eta,\xi)$ at $(\partial D)_1$ is given by
\ben
P(\eta,\xi)&=&\sum_{i=1}^4\sum_{k=1}^\infty{R^{2kn}|z|\frac{n}{2}(v_{k;i})^{-\frac{n}{2}-1}F\left(\frac{n}{2}+1,\frac{n}{2};n;\frac{u_{k,i}}{v_{k,i}}\right)}\left[R^{-2n}(N(\xi))^{-2n}\left(\frac{2|z|^2+t^2}{(|z|^4+t^2)^2}\right.\right.\\
&&\times\left.\left.(1+2R^{2(2k-2)}(|z|^2|\varsigma|^2-t\tau))-\frac{R^{2(2k-2)}}{|z|^4+t^2}(2|\varsigma|^2-t\tau)\right)\delta_{1i}+R^{-2n}((|z|^2+t^2)\right.\\
&&\left.+R^{2(2k-2)}(2|\varsigma|^2-t\tau))\delta_{2i}+R^{4k}(N(\xi))^{-2n}\left(\frac{2|z|^2+t^2}{(|z|^4+t^2)^2}(R^{4k}+2(|z|^2|\varsigma|^2-t\tau))\right.\right.\\
&&\left.\left.-\frac{1}{|z|^4+t^2}(2|\varsigma|^2-t\tau)\right)\delta_{3i}+R^{4k}(R^{4k}(|z|^2+t^2)+(2|\varsigma|^2-t\tau))\delta_{4i}\right],
\een where $\delta_{ai}$ denotes Dirac function of $\{i\}$. The Poisson kernel $P(\eta,\xi)$ at $(\partial D)_2$ is given by
\ben
P(\eta,\xi)&=&\sum_{i=1}^4\sum_{k=1}^\infty{-R^{(2kn-2)}|z|\frac{n}{2}(v_{k;i})^{-\frac{n}{2}-1}F\left(\frac{n}{2}+1,\frac{n}{2};n;\frac{u_{k,i}}{v_{k,i}}\right)}\left[R^{-2n}(N(\xi))^{-2n}\left(\frac{2|z|^2+t^2}{(|z|^4+t^2)^2}\right.\right.\\
&&\times\left.\left.(1+2R^{2(2k-2)}(|z|^2|\varsigma|^2-t\tau))-\frac{R^{2(2k-2)}}{|z|^4+t^2}(2|\varsigma|^2-t\tau)\right)\delta_{1i}+R^{-2n}((|z|^2+t^2)\right.\\
&&\left.+R^{2(2k-2)}(2|\varsigma|^2-t\tau))\delta_{2i}+R^{4k}(N(\xi))^{-2n}\left(\frac{2|z|^2+t^2}{(|z|^4+t^2)^2}(R^{4k}+2(|z|^2|\varsigma|^2-t\tau))\right.\right.\\
&&\left.\left.-\frac{1}{|z|^4+t^2}(2|\varsigma|^2-t\tau)\right)\delta_{3i}+R^{4k}(R^{4k}(|z|^2+t^2)+(2|\varsigma|^2-t\tau))\delta_{4i}\right].
\een
\begin{thm}
The Green's function and Poisson kernel which we have obtained above solves the Dirichlet boundary value problem for $D$ and the solution for BVP
\ben
L_0 u&=&f \;\text{in}\; D,\\
u&=&h \;\text{on}\; \partial D
\een is given by
$$u(\xi)=\int_D{G(\eta,\xi)f(\xi)dv(\xi)}+\int_{(\partial D)_1}{P(\eta,\xi)h(\xi)d\sigma(\xi)}+\int_{(\partial D)_2}{P(\eta,\xi)h(\xi)d\sigma(\xi)}$$ where $f$ and $h$ are circular functions.
\end{thm}
\section{GREEN'S FUNCTION FOR INFINITE STRIP}
In this section, $I$ will denote the infinite strip $\{\xi=[z',t']\in\H_n:0<t'<1\}$. Denote, by $H(t)$ the function of $(\eta,\xi),$ 
$$H(t)=a_0|C_{-t}|^{-n}F\left(\frac{n}{2},\frac{n}{2};n;\frac{|P|^2}{|C_{-t}|^2}\right),$$ where $C_{\pm t}$ and $P$ are defined as follows
\ben
C_{\pm t}&=&|z|^2+|z'|^2+i(t'\pm t)\\
P&=&2z.\bar{z}',
\een for $\xi=[z',t'] \in\H_n$ , $\eta=[z,t]\in \H_n$. A differential operator, whenever applied to function $H(t)$ will be with respect to the variable $\xi$.\\ Consider the following series $$\sum_{m=0}^\infty{H(2m-t)},\sum_{m=0}^\infty{H(2m+t)},\sum_{m=1}^\infty{H(-2m-t)},\sum_{m=1}^\infty{H(-2m+t)}.$$ We first show that the four series are uniformly convergent on compact neighbourhoods of $\xi$. For this firstly we show that the sequences of functions $$F\left(\frac{n}{2},\frac{n}{2};n;\frac{|P|^2}{|C_{2m-t}|^2}\right), F\left(\frac{n}{2},\frac{n}{2};n;\frac{|P|^2}{|C_{2m+t}|^2}\right),F\left(\frac{n}{2},\frac{n}{2};n;\frac{|P|^2}{|C_{-2m-t}|^2}\right), F\left(\frac{n}{2},\frac{n}{2};n;\frac{|P|^2}{|C_{-2m+t}|^2}\right)$$ are uniformly bounded.\\Firstly, consider the argument $u_{m,1}$ of the hypergeometric function $F\left(\frac{n}{2},\frac{n}{2};n;\frac{|P|^2}{|C_{2m-t}|^2}\right)$ i.e,
\ben
|u_{m,1}|=\left|\frac{|P|^2}{|C_{2m-t}|^2}\right|&=&\frac{4|z|^2|z'|^2}{||z|^2+|z'|^2+i(t'+2m-t)|^2}\\
&=&\frac{4|z|^2|z'|^2}{|z|^4+|z'|^4+2|z|^2|z'|^2+(t'+2m-t)^2}\\
&\leq&\frac{4|z|^2|z'|^2}{|z|^4+|z'|^4+2|z|^2|z'|^2}.\\
\text{So},\;1-\left|\frac{|P|^2}{|C_{2m-t}|^2}\right|&\geq&1-\frac{4|z|^2|z'|^2}{|z|^4+|z'|^4+2|z|^2|z'|^2}\\
&=&\frac{|z|^4+|z'|^4-2|z|^2|z'|^2}{|z|^4+|z'|^4+2|z|^2|z'|^2}\\
&=&\frac{(|z|^2-|z'|^2)^2}{|z|^4+|z'|^4+2|z|^2|z'|^2}.
\een We can choose a suitable compact neighbourhood of $\xi$ such that $|z|^2-|z'|^2>\epsilon$ for some $\epsilon>0$. So, we have $$1-\frac{|P|^2}{|C_{2m-t}|^2} \geq \frac{\epsilon^2}{|z|^4+|z'|^4+2|z|^2|z'|^2}>\epsilon_1,$$ for some $\epsilon_1>0$.\\ $\Rightarrow \frac{|P|^2}{|C_{2m-t}|^2}<1-\epsilon_1$ i.e, argument $u_{m,1}$ of hypergeometric function $F\left(\frac{n}{2},\frac{n}{2};n;u_{m,1}\right)$ is bounded away from 1. Therefore, $\left\{F\left(\frac{n}{2},\frac{n}{2};n;u_{m,1}\right)\right\}_{m=1}^\infty$ is uniformly bounded, say $$\left|F\left(\frac{n}{2},\frac{n}{2};n;u_{m,1}\right)\right|< S.$$ Similarly, it can be shown that
$\left\{F\left(\frac{n}{2},\frac{n}{2};n;|u_{m,i}|\right)\right\}_{m=0}^\infty,\;i=2,3,4,$ are uniformly bounded, where\\
$u_{m,2}=\frac{|P|^2}{|C_{2m+t}|^2},\;u_{m,3}=\frac{|P|^2}{|C_{-2m-t}|^2}\;\text{and}\;u_{m,4}=\frac{|P|^2}{|C_{-2m+t}|^2}.$\\ 
Next, consider the term $|C_{2m-t}|^{-n}$, $n>1$.
\ben
|C_{2m-t}|^{-n}&=&||z|^2+|z'|^2+i(t'+2m-t)|^{-n}\\
&=&(m)^{-n}\left|\frac{|z|^2+|z'|^2}{m}+i\left(\frac{t'-t}{m}+2\right)\right|^{-n}\\
&\rightarrow&0\; as\; m\rightarrow\infty,
\een on any compact neighbourhood of $\xi=[z,t]$, for fixed $\eta=[z',t']$.\\Consider
\ben
|H(2m-t)|&=&\left|F\left(\frac{n}{2},\frac{n}{2};n;\frac{|Q|^2}{|C_{2m-t}|^2}\right).a_0|C_{2m-t}|^{-n}\right|\\
&\leq&S.|a_0| |C_{2m-t}|^{-n}\\
&\rightarrow&0 \;as\; m\rightarrow\infty, 
\een on compact neighbourhood of $\xi$ for fixed $\eta$. Therefore, $\sum_{m=1}^\infty{H(2m-t)}$ is uniformly convergent on compact neighbourhoods of $\xi$.\\Similarly, $\sum_{m=0}^\infty{H(2m+t)},\sum_{m=0}^\infty{H(-2m-t)},\sum_{m=1}^\infty{H(-2m+t)}$ are uniformly convergent on compact neighbourhoods of $\xi$.\\Define
\begin{eqnarray}
G'(\eta,\xi)=\sum_{m=0}^\infty{H(2m+t)}-\sum_{m=0}^\infty{H(-2m-t)}+\sum_{m=1}^\infty{H(-2m+t)}-\sum_{m=1}^\infty{H(2m-t)}.
\end{eqnarray}For each $\eta$, $G'(\eta,\xi)$ is a well defined function. Now, we claim that $G'(\eta,\xi)$ works as a Green function for domain $I$ when applied to circular functions.
\begin{thm}
The function $G'(\eta,\xi)$ is a smooth function on $I=\{\xi=[z',t']\in\H_n:0<t'<1\}$ and satisfies the following.\\(i) $L_0G'(\eta,\xi)=\delta_\eta$.\\(ii) Limits of the function $G'(\eta,\xi)$ vanishes at the boundaries of infinite strip i.e, at $t'=0$ and $t'=1$.
\end{thm}
\begin{proof}
First note that for $m>0$, $H(2m+t),H(2m-t),H(-2m-t)\;\text{and}\;H(-2m+t)$ are all harmonic functions on $I$. Laplacian can be applied to the series so term by term and $$L_0G'(\eta,\xi)=L_0H(t)=\delta_\eta.$$ It can be easiy seen that as $t'\rightarrow 0,\;H(2m+t)\rightarrow H(-2m-t)$ and $H(-2m+t)\rightarrow H(2m-t)$ as, $|C_{2m+t}|^2 \rightarrow |C_{-2m-t}|^2$ and $|C_{-2m+t}|^2 \rightarrow |C_{2m-t}|^2$. Therefore,
$$\lim_{t'\rightarrow 0}\left(\sum_{m=0}^\infty{H(2m+t)}-\sum_{m=0}^\infty{H(-2m-t)}+\sum_{m=1}^\infty{H(-2m+t)}-\sum_{m=1}^\infty{H(2m-t)}\right)=0.$$ Moreover,
$$\lim_{t'\rightarrow 1}\left(\sum_{m=0}^\infty{H(2m+t)}-\sum_{m=0}^\infty{H(-2m-t)}+\sum_{m=1}^\infty{H(-2m+t)}-\sum_{m=1}^\infty{H(2m-t)}\right)$$
$$=\lim_{t'\rightarrow 1}(H(t)+H(2+t)+\ldots-H(-t)-H(-2-t)-\ldots+H(-2+t)+H(-4+t)+\ldots-H(2-t)-H(4-t)).$$ From the expression of $H(t)$, it can be easily seen that as $t'\rightarrow 1, H(t)\rightarrow H(-t)$ and for $m\geq 1$, $H(2m+t) \rightarrow H(-2m-t)$ and for $m\geq 2$, $H(2m-t) \rightarrow H(-2m+t).$\\Therefore, the last term is zero.\\
Hence,
$G'(\eta,\xi)$ given by $(6)$ is smooth Green's function for $I$ when applied to circular functions.
\end{proof}
The Dirichlet BVP similar to that in Theorem $2.2$ on I can be solved by obtaining a Poisson kernel on $I$. 
\section*{acknowledgements}
The authors are thankful to P. K Ratnakumar, Harish Chandra Research Institute, Allahabad, India, for his valuable discussion. The first author was supported by the Junior Research Fellowship of Council of Scientific and Industrial Research, India (Grant no. 09/045(1152)/2012-EMR-I).

\end{document}